\documentclass[12pt]{article}
\usepackage{amsmath,amsfonts,amssymb,amsthm}

\theoremstyle{definition}
\newtheorem{definition}{\it Definition}[section]
\newtheorem{example}[definition]{Example}
\newtheorem{remark}{Remark}

\theoremstyle{plain}
\newtheorem{theorem}{Theorem}[section]

\newtheorem{corollary}{Corollary}[section]

\providecommand{\prob}{\mathsf P}

\numberwithin{equation}{section}

\newcommand{\R}{\mathbb R}

\newcommand{\E}{\textsf E}

\newcommand{\Sub}{\mathrm{Sub}}

\newcommand{\Cov}{\textrm{Cov}}

\newcommand{\T}{\mathbf T}

\newcommand{\Var}{\textrm{Var}}

\begin{document}
\begin{center}
{\large\bf Investigation of Airy equations with random initial conditions}
 \end{center}

\begin{center}
{{\bf Lyudmyla Sakhno}}
\end{center}
\begin{center}
 {\it Department of Probability Theory, Statistics and Actuarial Mathematics\\ Taras Shevchenko National University of Kyiv,  Kyiv,  Ukraine}
\end{center}

\begin{abstract}
{The paper investigates properties of  mean-square solutions to the Airy equation with random initial data given by stationary processes. The result on the modulus of continiuty of the solution is stated and properties of the covariance function are described. Bounds for the distributions of the suprema of solutions under $\varphi$-sub-Gaussian initial conditions are presented. Several examples are provided to illustrate the results. Extension of the results to the case of fractional Airy equation is given.}


\end{abstract}

\section{Introduction}

Dispersive partial differential equations have been the topic of intensive studies for centuries, starting from the classical physics and going widely beyond. These equations serve to model various vibrating media such as waves on string, in liquids or in air. Probably the most familiar problems in which dispersive effects appear are the classical problems of water waves descriptions and one of the most famous equations is the Korteweg--de~Vries (KdV) equation derived to model the propagation of low amplitude long water waves in a shallow canal. Nowadays many versions and higher order generalizations of the KdV equation are of use in different areas including hydrodynamics, plasma physics, electrodynamics, in studies of electromagnetic and acoustic waves, waves in elastic media, turbulence, traffic flows, mass transport and others.

In the present paper we consider the Airy equation or linear Korteweg--de~Vries equation
\begin{equation}\label{intr1} 
\frac{\partial u}{\partial t} = - \frac{\partial^3 u}{\partial^3 x},\,\, t > 0, \,\,x \in \mathbb{R},
\end{equation}
subject to the random initial condition 
\begin{equation}\label{intr2}
u(0, x) = \eta(x), \,\, x \in \mathbb{R},
\end{equation}
with $\eta$ being a stationary stochastic process. 

Initial value problem \eqref{intr1}-\eqref{intr2} was treated in \cite{BKLO}, namely, the asymptotic behavior was analyzed for the rescaled solutions under weakly dependent stationary initial data. Note that in the recent literature such approach has been widely applied to study  rescaled  solutions  to the heat, fractional heat, Burgers and other equations with Gaussian and non-Gaussian initial conditions possessing weak or strong dependence (see, for example, papers \cite{AL2001, RMAA2001}  among many others).

The purpose of the present paper is to establish the upper bounds for the distribution of the supremum of solution to \eqref{intr1}-\eqref{intr2} under the assumption that the initial condition is given by a $\varphi$-sub-Gaussian process. These processes provide a natural generalization of Gaussian and sub-Gaussian ones and possess well described properties, which is important for applications. Theory developed for these processes allows to derive many useful  bounds for the distribution of various functionals of these processes (see, \cite{BK}). The present paper is close to the papers \cite{BKOS, KOSV2018, KOSV, HS} where higher order dispersive equations and the heat equation were studied under $\varphi$-sub-Gaussian initial conditions. 

Following \cite{BKLO}, we consider the solution to the random initial value problem \eqref{intr1}-\eqref{intr2} in the mean square sense and write its representation in a form of stochastic integral, and corresponding representation for its covariance function. The solution field is a stationary (of the second order) both in time and space variables.

We state the result on the modulus of continuity of solution and reveal several interesting properties of the covariance function of  solution. 
We show that the covariance function itself is a solution to a particular deterministic Airy equation and as such it inherits well known properties of solutions to Airy equations. In particular, many useful bounds can be written for the covariance function. 

It is important to note that from the statistical point of view,  the covariance of the solution presents an example of non-separable space-time covariance function. So, we reveal a convenient way to construct a space-time covariance with the use of another covariance model, which will possess a very transparent  physical interpretation and a lot of well described properties due to its representation as an oscillatory integral.

From the probabilistic consideration, we also deduce such a  general fact that if the Airy equation \eqref{intr1} is considered with a (nonrandom) initial condition $u(0,x)=u_0(x)$, $x\in \R$,  given by a real positive-definite kernel, then the solution $u(t, x)$, $(t, x)\in \R_+\times\R$, is a positive definite kernel as well. To the best of our knowledge, this interesting fact was not reflected before in the literature on the Airy equations.

Next, using the methods developed for $\varphi$-sub-Gaussian process, the bounds are obtained for the tails of the distribution of supremum of solution to \eqref{intr1}-\eqref{intr2}. 

The main condition for the bounds to hold is stated in terms of the spectral measure of the initial  data $\eta$ in \eqref{intr2}. It is shown that this condition is satisfied for several models where the initial data process $\eta$ is itself a solution to a stochastic differential equation. In particular, Mat\'ern model, Ornstein-Uhlenbeck and fractional Ornstein-Uhlenbeck processes can be considered to model random initial data.
 
 Finally, the extension of the results is outlined for the case of the fractional Airy equation introduces in \cite{MO}:
$ u_t = {\cal{D}}^\alpha_x u$, $t > 0$, $x \in \mathbb{R}$,
where ${\cal{D}}^\alpha_x$ represents the Riesz--Feller fractional derivative.

The paper is organized as follows. In Section 2, following \cite{BKLO}, we give the expression in the form of a stochastic integral for the mean square solution to the initial value problem   \eqref{intr1}--\eqref{intr2}. In Section 3 the result on the modulus of continuity of solution is stated. Section 4 describes properties of the covariance function of the solution.  Section 5 collects definitions and properties of $\varphi$-sub-Gaussian processes as a preparation for the study of the distribution of suprema of solutions to \eqref{intr1}--\eqref{intr2} under $\varphi$-sub-Gaussian initial conditions, which is done in Section 6. To illustrate the results stated, in Section 7 several examples of random initial data are given. Section 8 outlines the extension to the case of fractional Airy equation introduced in \cite{MO}. 

\section{Solution to the Airy equation with initial condition given by a stationary stochastic process}

Consider the Airy equation 
\begin{equation} \label{4.1}
\frac{\partial u}{\partial t} = - \frac{\partial^3 u}{\partial^3 x},\,\, t > 0, \,\,x \in \mathbb{R},
\end{equation}
subject to the random initial condition 
\begin{equation}\label{4.2}
u(0, x) = \eta(x), \,\, x \in \mathbb{R},
\end{equation}
where $\eta$ is a stochastic process satisfying the condition below.

\begin{description}
\item[A.]  $\eta(x), x \in \mathbb{R}$, is a real, measurable, mean-square continuous stationary (of the second order, that is, in a weak sense) centered stochastic process defined on a complete probablity space $(\Omega, \cal{F}, P)$. 
\end{description}

Let $B_\eta(x), x \in \mathbb{R},$ be a covariance function of the process $\eta(x), x \in \mathbb{R}$, with the spectral  representation 
\begin{equation}\label{4.3}
B_\eta(x) = \int_\mathbb{R} \cos(\lambda x) dF(\lambda),
\end{equation}
where $F(\lambda), \lambda \in \mathbb{R},$ is a spectral measure, and for the process itself we can write the spectral representation
\begin{equation}\label{4.4} 
\eta(x) = \int_\mathbb{R} e^{i \lambda x} Z(d\lambda).
\end{equation}

The stochastic integral \eqref{4.4} is considered as $L_2(\Omega)$ integral. The orthogonal complex-valued random measure $Z$ is such that $\E|Z(d\lambda)|^2 = F(d\lambda)$. 

Following \cite{BKLO}, we can write the representation of mean square solution to the problem \eqref{4.1}--\eqref{4.2} and the expression for its covariance function.

Consider the process $\{u(t, x), t > 0, x \in \mathbb{R}\}$ defined by 
\begin{equation}\label{4.5}
u(t, x) = \int_\mathbb{R} g(t, x - y)\eta(y) dy,
\end{equation}
where the function $g$ is the fundamental solution to equation \eqref{4.1}:
\begin{equation}\label{4.6}
g(t, x) = \frac{1}{2\pi} \int_\mathbb{R} e^{-i\alpha x-i\alpha^3 t}\, d\alpha=\frac{1}{\pi} \int_0^\infty \cos(\alpha x+\alpha^3 t)\, d\alpha=\frac{1}{\sqrt\pi\sqrt[3]{3t}}Ai\Big(\frac{x}{\sqrt[3]{3t}}\Big),\, t > 0,\, x \in \mathbb{R},
\end{equation}
and
\begin{equation*}
Ai(x) = \frac{1}{\sqrt\pi}\int_0^\infty \cos\Big(\alpha x+\frac{\alpha^3}{3}\Big)\, d\alpha, \, x \in \mathbb{R},
\end{equation*}
is the Airy function of the first kind.

In view \eqref{4.4}, the process  \eqref{4.5} can be written in the following form
\begin{equation}\label{4.7}    
u(t, x) = \int_\mathbb{R} \exp\Big\{ i\lambda x +  i \lambda^3 t \Big\} Z(d\lambda).
\end{equation}
The process \eqref{4.7} can be interpreted as the mean-square or $L_2(\Omega)$ solution to the Cauchy problem \eqref{4.1}--\eqref{4.2} (see \cite{BKLO}). 

From  the representation \eqref{4.7} the covariance of the field $u$ can be calculated:
\begin{eqnarray}\label{4.88}
Cov \big( u(t, x), u(s, y)\big) &=& \int_\mathbb{R} \exp(i \lambda (x - y) +i \lambda^3(t - s) )dF(\lambda) \notag\\&=& \int_\mathbb{R} \cos(\lambda (x - y) +i \lambda^3(t - s) )dF(\lambda)\\
&:=& B(t-s, x-y)\notag.
\end{eqnarray}

From the expression \eqref{4.88} one can see that the random field $u$ is stationary with respect to time and space variables. For a fixed $t$ we have
$$
Cov \big( u(t, x), u(t, y)\big) = \int_\mathbb{R} \exp(i \lambda (x - y) )dF(\lambda) =B_\eta(x-y),
$$
that is, the covariance of the solution $u(t, \cdot)$ considered at any fixed time $t$ coincides with the covariance function of the initial process $\eta$.

Note that one of the commonly used approaches of studing PDEs with stationary initial conditions is via the second order analysis, that is, by considering their mean square solutions represented by stochastic integrals. 
Such approach takes its origins in the paper by Rosenblatt \cite{R} where the heat equation with stationary initial condition was treated and mean square representation for the solution was given. In the recent literature this approach is widely used for various classes of PDEs with random initial conditions, in particular, for studying their rescaled solutions, see, e.g., \cite{AL2001, BKLO, HS, RMAA2001}, to mention only few, see also references therein.

\section{Modulus of continuity of the solution}

In this section we state the general result on the modulus of continuity in the mean square of the field \eqref{4.7}. This result is of interest by itself, but will also be used  further in Section 6 for evaluation of the distribution of suprema of solutions. Denote $K=[a,b]\times[c,d]$ for arbitrary $a\ge 0$, $b,c,d \in \R$.

\begin{theorem}\label{th4.1} Let $u(t, x), t > 0, x \in \mathbb{R},$ be the random field given by \eqref{4.7} and assumption A hold. Suppose that 
 fore some $\beta \in (0, 1]$
\begin{equation}\label{4.8}
\int_\mathbb{R} \lambda^{6\beta} F(d \lambda) < \infty,
\end{equation}
then 
\begin{equation}\label{4.9}
\sigma(h) := \sup_{\substack {(t,x), (s,y) \in K: \\ |t - s| \le h, |x - y| \le h}} \Big( \E (u(t, x) - u(s, y))^2\Big)^{1/2} \le  c(\beta) h^\beta,
\end{equation}
where 
\begin{equation}\label{4.10}
c(\beta) = 2^{1-\beta}\Big( \int_\mathbb{R} (\lambda + \lambda^{3})^{2\beta} F(d\lambda)\Big)^{1/2}.
\end{equation}
\end{theorem}

\begin{proof} We have
\begin{equation}\label{4.15}
\E\Big( u(t, x) - u(s, y)\Big)^2 = \int_\mathbb{R} |b(\lambda)|^2 F(d\lambda),
\end{equation}
where 
$$b(\lambda) = e^{i(\lambda x+\lambda^3 t)}  - e^{i(\lambda y+\lambda^3 s)}.$$
By the direct calculations we obtain:
$$|b(\lambda)|^2 = 4\sin^2 \left(\frac{\lambda (x - y)+\lambda^3 (t - s)}{2}\right).$$
For $|t - s| \le h$ and $|x - y| \le h$ we can write for any $\beta \in (0, 1]$:
\begin{equation}\label{4.16}
4\sin^2 \left(\frac{1}{2}(\lambda (x - y)+\lambda^3 (t - s))\right) \le 4 \min \Big(\frac{h}{2}|\lambda+\lambda^3|, 1\Big)^{2} \le 4 \frac{(h(\lambda+\lambda^3))^{2\beta}}{2^{2\beta}}
\end{equation}
which implies the estimate
\begin{equation}\label{4.15}
\Big(\int_\mathbb{R} |b(\lambda)|^2 F(d\lambda)\Big)^{1/2}\le 2^{1-\beta}h^{\beta}\Big( \int_\mathbb{R} (\lambda + \lambda^{3})^{2\beta} F(d\lambda)\Big)^{1/2}.
\end{equation}
Therefore, under condition \eqref{4.8} we can  write the bound \eqref{4.9}.

\end{proof}

\section{Closer look at the covariance function of the solution}

In this section we discuss the covariance function of the solution and reveal several interesting facts and properties coming from its representation.

Firstly, we note that the covariance function $B(t, x)$ of the random solution field $u(t, x)$ to the initial value problem \eqref{4.1}--\eqref{4.2} is itself a solution to the  deterministic initial value problem 
\begin{eqnarray} \label{44.1}
&&\frac{\partial B}{\partial t} = - \frac{\partial^3 B}{\partial^3 x},\,\, t > 0, \,\,x \in \mathbb{R},\\
&&\label{44.2}
B(0, x) = B_\eta(x), \,\, x \in \mathbb{R}
\end{eqnarray}
(where $B_\eta$ is supposed to be sufficiently smooth and decaying) and as such, it possesses all the properties of a solution to the Airy equation. Theory of dispersive equations and, in particular, KdV and linear KdV equations, is well developed and presents many interesting and important results on the solutions.

We collect here some properties of the covariance function $B(t, x)$ due to the results on Airy equations available in the literature (see, e.g., \cite{Tao}).

Suppose that there exists the spectral density $f(\lambda)\in L_1(\R)$ of the initial value process $\eta$, so that the covariance has the representation
\begin{equation*}
B_\eta(x)=\int_\mathbb{R} e^{i \lambda x } f(\lambda)\,d\lambda,
\end{equation*}
supposing  $B_\eta\in L_2(\R)$, we can write the inverse transform
$
f(\lambda)=\frac{1}{2\pi}\int_\mathbb{R} e^{-i \lambda x } B_\eta(x)\,dx.
$
In such a case, due to the results of Section 2, we have
\begin{equation*}
B(t,x)=\int_\mathbb{R} e^{i \lambda x +i \lambda^3 t} f(\lambda)\,d\lambda,
\end{equation*}
which exactly gives the solution to the problem \eqref{44.1}--\eqref{44.2}.
 As the solution to the Airy equation, $B(t, x)$ ``disperses'' as $t\to\infty$ and various so-called dispersive estimates for $B(t, x)$ can be written, in particular, 
\begin{eqnarray*}
&&\|B(t,x)\|_{L^\infty_x}\le c|t|^{-1/3}\|B_\eta\|_{L^1},
\end{eqnarray*} 
and also (under appropriate conditions on $B_\eta$) more general  Airy--Strichartz inequalities hold for  $B(t,x)$, of which we present here the simplest ones:
\begin{eqnarray*}
&&\|B(t,x)\|_{L^8_x}\le c|B_\eta\|_{L^2}\\
&&\| \partial_x B(t,x)\|_{L^\infty_x L^2_t}\le c\|B_\eta\|_{L^2}.
\end{eqnarray*}
For more detail, more estimates and historical background see, for instance, \cite{KPV}.

On the other hand, along with the ``dispersive'' behavior,  $B(t,x)$ possesses some conserved quantities, namely, $\int_\mathbb{R} B(t, x)\,dx$ is constant in time, or conserved:
$$\int_\mathbb{R} B(t, x)\,dx=\int_\mathbb{R} B_\eta(x)\,dx, $$
and $B(t,x)$  also preserves the $L_2$ norm, that is, $\int_\mathbb{R} B^2(t, x)\,dx$ is constant in time: 
$$\|B(t, x)\|_{L^2_x}=\|B_\eta\|_{L^2}.$$

Considering $B(t,x)$ from the statistical point of view, we see that it gives an example of nonseparable space-time covariance function, constructed with the use of another covariance model, possessing a very transparent  physical interpretation and a lot of well described properties due to its representation as an oscillatory integral.

We mention one more interesting fact coming as a feedback from the probabilistic consideration of the Airy equation $\partial_t u =-\partial_{xxx}u$. Namely, we conclude that if the initial condition $u(0,x)=u_0(x)$, $x\in \R$, is given by a real positive-definite kernel, that is, $\sum_{i=1}^n c_i c_j u_0(x_i-x_j)\ge 0$ for all $n\in N$, $c_i\in \R$, $x_i\in \R$, then the solution $u(t, x)$, $(t, x)\in \R_+\times\R$, is a positive definite kernel as well: $\sum_{i=1}^n c_i c_j u(t_i-t_j, x_i-x_j)\ge 0$ for all $n\in N$, $c_i\in \R$, $(t_i, x_i)\in \R_+\times\R$  (since it can be seen as a covariance function of the random field representing the solution to the corresponding random initial value problem). To the best of our knowledge, this fact was not reflected before in the literature on the Airy equations.

\section{$\varphi$-sub-Gaussian stochastic processes to be used as initial conditions}

The main theory for the spaces of $\varphi$-sub-Gaussian random variables and stochastic processes was presented in \cite{BK, GKN, KO} and has been further developed in  numerous recent studies. Such spaces can be considered as exponential type Orlicz spaces of random variables and provide generalizations of Gaussian and sub-Gaussian random variables and processes (see, \cite[Ch.2]{BK}). To make the paper selfcontained, we present definitions and facts needed in our study.

\begin{definition}\cite{GKN,KO}  A continuous even convex
function $\varphi$ is called {\it an Orlicz N-function} if
$\varphi(0)=0$, $\varphi(x)>0$ as $x\neq0$ and 
$\lim\limits_{
        x\to 0}\frac{\varphi(x)}{x}=0$, $\lim\limits_{
x\rightarrow \infty 
}\frac{\varphi(x)}{x}=\infty.$
\end{definition}

\noindent {\it Condition Q.} Let  $\varphi$ be an N-function which satisfies 
    $\lim \inf\limits_{
            x\to 0 
    }\frac{\varphi(x)}{x^2}=c>0,$ where
 the case $c=\infty$ is possible.

\begin{definition}\cite{GKN,KO} \label{def}
Let $\varphi$ be an $N$-function satisfying condition $Q$ and $\{\Omega, L, \prob \}$ be a standard probability space. The random variable $\zeta$ is  $\varphi$-sub-Gaussian, or belongs to the space $\Sub_{\varphi} (\Omega)$,   if $\E \zeta=0,$ $ \E \exp\{\lambda\zeta\}$ exists for all $\lambda\in\R$ and there exists a constant  $a>0$ such that the following inequality holds for all $\lambda\in\R$
$$\E \exp\{\lambda \zeta\} \leq \exp \{\varphi(\lambda a)\}.$$
The random process $X(t)$, $t\in T$, is called  $\varphi$-sub-Gaussian  if the  random variables
$\{X(t), t\in T\}$ are  $\varphi$-sub-Gaussian.
\end{definition}

The space $ \Sub_\varphi(\Omega)$ is a Banach space with respect to
the norm (see \cite{GKN,KO}):
$$ \tau_\varphi (\zeta) = \inf\{a > 0: \E \exp \{\lambda \zeta \} \leq
\exp\{ \varphi (a \lambda)\}.$$

\begin{definition}\cite{GKN,KO} 
The function $\varphi^*$
defined by
 $\varphi^*(x)=\sup_{\substack{
       y\in \R}}(xy-\varphi(y))$
 is called {\it the Young-Fenchel transform (or convex conjugate)} of the function $\varphi$.
\end{definition}

The function $\varphi^*$ (known also as the Legendre or Legendre-Fenchel transform) plays an important role in the theory of $\varphi$-sub-Gaussian random variables and processes and involved in estimates for 
 `tail' probabilities, distributions of suprema and other functionals of these processes. If $\zeta$ is a $\varphi$-sub-Gaussian random variable, then for all $u>0$ we have
\begin{equation}\label{tail}
P\{|\zeta|>u\}\le 2\exp\left\{-\varphi^*\left(\frac{u}{\tau_\varphi(\zeta)}\right)\right\}.
\end{equation}
Moreover, it is stated in \cite{BK} (see, Corollary 4.1, p. 68) that a random variable
$\zeta$ is a $\varphi$-sub-Gaussian if and only if  $\E\zeta=0$ and there exist constants $C>0$, $D>0$ such that 
\begin{equation}\label{tail1}
P\{|\zeta|>u\}\le C\exp\left\{-\varphi^*\left(\frac{u}{D}\right)\right\}.
\end{equation}

As one can see, the property of $\varphi$-sub-Gaussianity can be characterized in a double way: by introducing a bound on the exponential moment of a random variable as prescribed by  Definition \ref{def}, or by the tail behavior of the form \eqref{tail} or \eqref{tail1}, which is even more essential from the practical point of view.

The class of $\varphi$-sub-Gaussian random variables is rather wide and comprises, for example, centered compactly supported distributions, reflected Weibull distributions, centered bounded distributions, Gaussian, Poisson distributions. In the case when $\varphi=\frac{x^2}{2}$, the notion of $\varphi$-sub-Gaussianity reduces to the classical sub-Gaussianity. Various  classes of $\varphi$-sub-Gaussian processes and fields were studied, in particular,  in \cite{HS, KO2016, KOP2015, KOSV2018, KOSV} (see also references therein).

Let us consider the metric space $(\T, \rho)$,  $\T= \{ a_i \le t_i \le b_i, i = 1,2 \}$, $\rho(t, s) = \max_{ i = {1,2}} |t_i - s_i|$  and
$X(t)$, $t\in\T$,
be a $\varphi$-sub-Gaussian process.
Introduce the following conditions.

\medskip
\noindent {\bf B.1.}  $\varepsilon_0= \sup_{t\in\T}\tau_\varphi(X(t))<\infty$.

\noindent {\bf B.2.}
The  process $X$ is separable on the space $(\T,\rho)$.

\noindent {\bf B.3.} For $c > 0$ and $0 < \beta \le 1$ the estimates holds:
$\sup_{\rho(t,s)<h}\tau_\varphi({X(t)-X(s)})\le c h^\beta.$


\medskip

\begin{theorem}{\rm(\cite{HS})} \label{Cor1.2} Let for a $\varphi$-sub-Gaussian  process $X(t)$, $t\in\T$, conditions B.1-B.3 hold and $\sigma(h) = c h^\beta$ with $c > 0, 0 < \beta \le 1$. Let $\varkappa=\max\{b_i-a_i, i=1,2\}$. Then for any $\theta \in (0, 1)$ and any $u > 0, \lambda > 0$
$$\E\exp\Big\{\lambda \sup_{t \in \T}|X(t)| \Big\} \le 2 \exp \Big\{ \varphi\Big(\frac{\lambda \varepsilon_0}{1 - \theta} \Big)\Big\} A_1(\theta\varepsilon_0),$$
$$P \Big\{  \sup_{t \in \T}|X(t)| \ge u\Big\} \le 2 \exp \Big\{ - \varphi^* \Big(\frac{u (1 - \theta)}{\varepsilon_0} \Big)\Big\} A_1(\theta\varepsilon_0),$$
where 
$$A_1(\theta \varepsilon_0) = 2^{\frac{4}{\beta} -1} \Big( \frac{\varkappa^2 c^{2/\beta} 2^{2(2/\beta -1)}}{(\theta\varepsilon_0)^{2/\beta}} + 1\Big).$$
\end{theorem}

\begin{remark}{ Theorem \ref{Cor1.2} and more general results, with  more general bounds on the increments in assumption B.3,   have been  applied in the literature in different contexts, in particular, in \cite{KO2016} for developing  uniform approximation schemes for $\varphi$-sub-Gaussian processes, in \cite{HS, KOSV} for studying partial differential equations with random initial condition, in \cite{Sakhno2022} for evaluation of suprema of spherical random fields. Such theorems allow to calculate the bounds for the distribution of suprema of $\varphi$-sub-Gaussian  process in the closed form, which is important from the practical point of view.}
\end{remark}

We will need some further properties of $\varphi$-sub-Gaussian variables.

\begin{definition}
\cite{KK} A family $\Delta$ of $\varphi$-sub-Gaussian random variables is called strictly $\varphi$-sub-Gaussian if
there exists a constant $C_\Delta $ such that for all countable sets $I$ of random
 variables ${\zeta_i}\in \Delta$, $i\in I$, the  inequality holds:
$
\tau_\varphi \big( \sum_{i \in I} \lambda_i \zeta_i \big) \leq
 C_\Delta \big( \E \left( \sum_{i \in I}\lambda_i \zeta_i \right)^2\big)^{1/2}.$ Random process $\zeta(t)$, $t\in T$, is called
strictly $\varphi$-sub-Gaussian  if the family of random variables
$\{\zeta(t), t\in T\}$ is strictly $\varphi$-sub-Gaussian.
\end{definition}

\begin{example} \cite{KK} Let $\xi_k$, $k = \overline{1, \infty}$ be  independent $\varphi$-sub-Gaussian random variables 
 and $\varphi$ be such  that $\varphi (\sqrt{x})$ is convex.
 If $\tau_\varphi (\xi_n ) \leq C ( \E \xi_k^2 )^{1/2}$, $ C> 0$, and for
 all $t \in T$ the series
  $\sum\limits_{k=1}^\infty \E\xi_k^2 \varphi_k^2 (t)$ converges,
  then the series $\sum\limits_{k=1}^\infty \xi_k \varphi_k (t)$,  $t \in T,$ is strictly
$\varphi$-sub-Gaussian random process with determining constant $C.$
\end{example}
\begin{example}\label{ex_kernel}\cite{KK} Let $K$ be a deterministic kernel and 
$X(t) =\int_T K(t,s) \,\,d\xi(s),
$
where $\xi(t)$, $t\in T$,  is a strictly $\varphi$-sub-Gaussian  process
and the integral is defined in the mean-square sense. Then
$X(t)$, $t\in T$, is  strictly $\varphi$-sub-Gaussian 
process with the same determining constant.
\end{example}

\section{Distribution of suprema of solutions under  stationary $\varphi$-sub-Gaussian initial conditions}

Consider the initial value problem \eqref{4.1}--\eqref{4.2}, where the process $\eta$ is strictly $\varphi$-sub-Gaussian and satisfies condition A. Suppose that the solution  $u(t, x)$  is considered in the domain $K=\{(t,x): a \le t \le b, c \le x \le d\}$.  Denote $\widetilde\varepsilon_0 = \sup_{\substack {(t, x)\in K }} \tau_\varphi(u(t, x))$, $\varkappa = \max(b - a, d - c)$, $\tilde\theta=\widetilde\sigma(\varkappa)/\widetilde\varepsilon_0$, where $\widetilde\sigma$ is defined in \eqref{ws} below.

\begin{theorem}\label{th6.1}
Let $u(t, x), (t,x)\in K$, be a separable modification of the stochastic process given by \eqref{4.7}, the process $\eta$ be strictly $\varphi$-sub-Gaussian with the determining constant $c_\eta$ and assumption A hold. Suppose that for
 for some $\beta \in (0, 1]$ condition \eqref{4.8} holds.

Then:

1)
\begin{equation}\label{ws}
\sup_{\substack {|t - s| \le h, |x - y| \le h}} \tau_\varphi(u(t, x) - u(s, y)) \le \widetilde\sigma(h):= c_\eta c(\beta) h^\beta,
\end{equation}

where $c(\beta)$ is given by formula \eqref{4.10};

2) for all $0 < \theta < \min(\tilde\theta,1)$ and $v > 0$ the following bound for the distribution of 

supremum holds:
\begin{equation}\label{sup}
P\Big\{\sup_{\substack { (t,x)\in K}} |u(t, x)| > v \Big\} \le 2\exp\Big\{ -\varphi^* \Big(\frac{v(1 - \theta)}{\widetilde\varepsilon_0}\Big) \Big\} \widetilde A_1(\theta \widetilde\varepsilon_0),\end{equation}

where
\begin{equation}\label{Asup} 
\widetilde A_1(\theta \widetilde\varepsilon_0) = 2^{\frac{4}{\beta} -1} \Big(\frac{(\varkappa c_\eta c(\beta)^{1/\beta})^2 4^{2/\beta - 1}}{(\theta \widetilde\varepsilon_0)^{2/\beta}} + 1 \Big);\end{equation}

3) for any $p \in (0, 1)$ and any $h > 0$
\begin{equation} \label{supincr}
P\Big\{ \sup_{\substack {|t - s| \le h,\\ |x - y| \le h}}|u(t, x) - u(s, y)| > v \Big\}  \le 2^{4/\beta} \exp\Big\{-\varphi^*\Big(\frac{v  (1 - p)^2}{ c_\eta c(\beta)h^\beta (3 - p)}\Big)\Big\} \Big[ \frac{2^{4/\beta-2}\varkappa^2}{p h^2} + 1 \Big]. \end{equation}
\end{theorem}

\begin{proof}
The process $u(t, x)$ is strictly $\varphi$-sub-Gaussian with the determining constant $c_\eta$. Therefore, in view of Theorem \ref{th4.1} we can write 
\begin{equation}\label{4.14}
\sup_{\substack {|t - s| \le h, |x - y| \le h}} \tau_\varphi (u(t, x) - u(s, y)) \le c_\eta \sup_{\substack {|t - s| \le h, |x - y| \le h}} \Big( \E (u(t, x) - u(s, y))^2\Big)^{1/2}\le c_\eta c(\beta) h^\beta.
\end{equation}
The assertion 2) of the theorem follows from Theorem \ref{Cor1.2}.
Assertion 1) gives the validity of condition B.3. 
So, it is left to check that $\widetilde \varepsilon_0 =  \sup_{\substack { (t,x)\in K}} \tau_\varphi(u(t, x)) < \infty$.
We have indeed 
\begin{align*}
\widetilde \varepsilon_0 \le c_\eta \Big(\E |u(t, x)|^2 \Big)^{1/2} 
= c_\eta \Big(\int_\mathbb{R}  F(d \lambda) \Big)^{1/2} < \infty 
\end{align*}
Assertion 3) follows from Theorem 3 in \cite{HS}.
\end{proof}

\begin{corollary}\label{cor}
Let  the process $\eta$ be Gaussian. Then
\begin{equation}\label{supGauss}
P\Big\{\sup_{\substack { (t,x)\in K}} |u(t, x)| > v \Big\} \le 2\exp\Big\{ -\frac{v^2(1 - \theta)^2}{2\widehat\varepsilon_0} \Big\} \widehat A_1(\theta \widehat\varepsilon_0),\end{equation}
where 
$\widehat\varepsilon_0=(B_\eta(0))^{1/2}=(\int_\R dF(\lambda))^{1/2}$
 and $\widehat A_1$ is given by formula \eqref{Asup} with $c_\eta=1$. 
\end{corollary}

\begin{example}
An important natural generalization of Gaussian processes is obtained with $\varphi(x)=\frac{|x|^{\alpha}}{\alpha}$, $1 <\alpha\leq 2$.  For this case $\varphi^*(x)=\frac{|x|^{\gamma}}{\gamma}$,  where $\gamma \ge 2$,  and $\frac{1}{\alpha}+\frac{1}{\gamma}=1$. For such $\varphi$-sub-Gaussian initial data the exponential term in \eqref{sup} takes the form $\exp \big\{-  (u^\gamma (1-\mu )^\gamma)/(\gamma\tilde\varepsilon_0^\gamma)  \big\}$. We can also conclude from assertion 3) of Theorem \ref{th6.1} that the solution $u$ is sample continuous with probability 1. Indeed, in this case we have that the right hand side of the formula \eqref{supincr} tends to 0 for $h\to 0$, then 
$P\Big\{ \sup_{\substack {|t - s| \le h, |x - y| \le h}}|u(t, x) - u(s, y)| > v \Big\} \to 0$. Therefore, as $h\to 0$, $\sup_{\substack {|t - s| \le h, |x - y| \le h}}|u(t, x) - u(s, y)| \to 0$ in probability, but also (due to the monotonicity of the supremum) with probability 1. 
\end{example}
\begin{remark} 
 In \cite{KOSV} similar results on the distribution of suprema were obtained  for solutions to higher-order linear dispersive equations with harmonizable $\varphi$-sub-Gaussian initial data. Solutions therein were considered as classical solutions, that is, satisfying the corresponding equations with probability 1, under the appropriate set of conditions. Considering in the present paper a simpler case of the Airy equation with stationary initial conditions, we use the approach via second order analysis and treat the solutions in the mean square sense. The bounds and conditions are presented in such an explicit form, which is convenient for practical applications. In the next section we present several models for the initial data process $\eta$, for which conditions of the stated results hold. 
\end{remark}

\section{Examples}

In this section we present several examples of processes which can be used as initial conditions. For these processes the condition \eqref{4.8} is satisfied, the constant $c(\beta)$ can be calculated in the closed form and the estimate \eqref{supGauss} holds. 

\medskip 
1. As for the first example, we consider  the well known and popular in various applications Mat\'ern model.

Let $\eta(x)$, $x\in \R$ be a Gaussian stochastic process with the spectral density
\begin{equation}\label{7.1}
f(\lambda)=\frac{\sigma^2}{(1+\lambda^2)^{2\alpha}}, \ \lambda\in\R.
\end{equation}
The corresponding covariance function is of the form:
\begin{equation}\label{7.2}
B_\eta(x)=\frac{\sigma^2}{\sqrt\pi \Gamma(2\alpha)}\Big(\frac{|x|}{2}\Big)^{2\alpha-1/2}K_{2\alpha-1/2}(|x|), \ x\in\R,
\end{equation}
where $K_\nu$ is the modified Bessel function of the second kind, in particular,  $K_{1/2}(x)=\sqrt{\frac{\pi}{2x}}e^{-x}$. Mat\'ern class of covariances \eqref{7.2} has a parameter $\nu=2\alpha-1/2>0$ that controls the level of smoothness of the stochastic process.
Mat\'ern class comprises a broad range of covariances, in particular, exponential covariance, which we consider in our next example.

The Gaussian stochastic process with the above covariance and spectral density can be obtained as a solution to the following factional partial differential equation:
\begin{equation}\label{fracPDE}
\Big(1-\frac{d^2}{dx^2}\Big)^{\alpha}\eta(x)=w(x), \ x\in\R,
\end{equation}
with $w$ being a white noise: $\E w(x)=0$ and $\E w(x)w(y)=\sigma^2\delta(x-y)$ (see, for example, \cite[Theorem 3.1]{DOS}).

Note, that Mat\'ern model is extremely popular in spatial statistics and in modeling random fields in various applied areas (with corresponding adjustment of the above covariance for $n$-dimensional case). The relation between the spatial Mat\'ern covariance model and the stochastic partial differential equation over $\R^n$  
\begin{equation*}
(\mu-\Delta)^{\alpha}\eta(x)=w(x), \ x\in\R^n,
\end{equation*}
was established by Whittle in 1963 and since then has been revisited and exploited by many authors in applied and theoretical contexts.

Consider the initial value problem \eqref{4.1}--\eqref{4.2} with initial data $\eta$ represented as solution to the equation \eqref{fracPDE}.
Due to the form of the spectral density \eqref{7.1}, we are able to calculate the constant $c(\beta)$ which is given by \eqref{4.10} and appears in the estimates  \eqref{4.9}, \eqref{ws}, \eqref{sup}, \eqref{Asup}.
We have 
$$
\int_\R \lambda^{2\beta}(1+\lambda^2)^{2\beta-2\alpha}=\int_0^\infty \frac{t^{\beta+1/2-1}}{(1+t)^{\beta+1/2+2\alpha-3\beta-1/2}}\, dt = {\cal {B}}(\beta+1/2, 2\alpha-3\beta-1/2),
$$
where ${\cal B}$ is Beta-function, $\beta\in(0,1]$, $2\alpha-3\beta-1/2>0$ and we used the formula $\int_0^\infty \frac{t^{\mu-1}}{(1+t)^{\mu+\nu}}\, dt = {\cal {B}}(\mu, \nu)$.

Therefore, in this case we obtain $c(\beta)=2^{1-\beta}({\cal {B}}(\beta+1/2, 2\alpha-3\beta-1/2))^{1/2}$. In particular, having in \eqref{7.1} $\alpha>1$ and choosing $\beta=1/2$ we get $c(1/2)=1/(\alpha-1)$.

\medskip
2. Consider the stationary Gaussian Ornstein-Uhlenbeck process $\eta$ defined by the equation
\begin{equation}\label{OU}
d\eta(x)=-\eta(x)dx+\gamma d w(x), \ x\in\R,
\end{equation}
where $w$ is the Brownian motion or Wiener process such that $\E w(x)=0$, $\Var w(x)=|x|$, $x\in\R$.
Stationary Gaussian solution to \eqref{OU} has the following covariance function and spectral density:
$$
B_\eta(x)=\frac{\gamma^2}{2}e^{-|x|}, \ x\in\R, \ f(\lambda)=\frac{\gamma^2}{2\pi(1+\lambda^2)}, \ \lambda\in\R.
$$
These covariance and spectral density are particular cases of those considered in the previous example, and the calculations for $c(\beta)$ are valid as well.

\medskip
3. Ornstein--Uhlenbeck equation driven by a  fractional Brownian motion. Consider the linear stochastic differential equation
\begin{equation}\label{fOU}
d\eta(x)=-\eta(x)dx+d W_H(x), \ x\in\R,
\end{equation}
where $W_H$, $\frac{1}{2}<H<1$, is a  fractional Brownian motion, that is, a zero mean Gaussian   process with $W_H(0) = 0$, stationary increments and covariance
$$
\Cov(W_H(x), W_H(y)
=\frac{c}{2}\big(
|x|^{2H} + |y|^{2H}-|x-y|^{2H}\big), \ x, y \in \R,
$$
where $c=\Var W_H(1)$.
One can show that there exists
a unique continuous solution of equation \eqref{fOU} in the form
$$
\eta(x)=\int_{-\infty}^x e^{-(x-y)}\, d W_H(y), x \in \R,
$$
which is a stationary Gaussian process with the spectral density
$$
f(\lambda)=\frac{\sigma^2}{1+\lambda^2}|\lambda|^{1-2H},\ \lambda \in \R,
$$
where $\sigma^2=c\Gamma(2H+1)\sin(\pi H)/(2\pi)$ (see, e.g., \cite{AL2019}). Similar to the example 1, we calculate
$$
\int_\R \frac{\lambda^{2\beta}(1+\lambda^2)^{2\beta}}{1+\lambda^2}|\lambda|^{1-2H}={\cal {B}}(\beta+1-H, H-3\beta),
$$
where we should choose $\beta<H/3$. Therefore, 
 $c(\beta)=2^{1-\beta}(\sigma^2 {\cal {B}}(\beta+1-H, H-3\beta))^{1/2}$.

\section{Extension to case of the fractional Airy equation}

In this section we outline how the results of the previous sections can be extended to the case of fractional Airy equation.

Following \cite{MO}, let us consider the fractional extension (in the space variable) of equation \eqref{4.1} of the following form:
\begin{equation} \label{8.1}
\frac{\partial u}{\partial t} = {\cal{D}}^\alpha_x u,\,\, t > 0, \,\,x \in \mathbb{R},
\end{equation}
where ${\cal{D}}^\alpha_x$ represents the Riesz--Feller fractional derivative. ${\cal{D}}^\alpha_x$ is defined by means of its Fourier transform
$$
{\cal{F}}\{{\cal{D}}^\alpha_x f(x)\}(\gamma)=-|\gamma|^\alpha e ^{\frac{i\pi}{2}sgn(\gamma)}{\cal{F}}\{f(x)\}(\gamma)
$$
with the notation ${\cal{F}}\{f(x)\}(\gamma)=\int_{\R}e^{i\gamma x}f(x)dx$.

The fundamental solution to the equation \eqref{8.1} can be represented in the form (\cite{MO}):
\begin{equation}\label{8.6}
g_\alpha(t, x) = \frac{1}{2\pi} \int_0^\infty  \cos(\gamma x+|\gamma|^\alpha t)\, d\gamma=\frac{1}{(\alpha t)^{1/\alpha}}Ai_\alpha\Big(\frac{x}{(\alpha t)^{1/\alpha}}\Big),\, t > 0,\, x \in \mathbb{R},
\end{equation}
where $Ai_\alpha$ is the generalized Airy function
\begin{equation*}
Ai_\alpha(x) = \frac{1}{\pi}\int_0^\infty \cos\Big(\gamma x+\frac{\gamma^\alpha}{\alpha}\Big)\, d\gamma, \, x \in \mathbb{R},\ \alpha>1.
\end{equation*}
It was shown in \cite{MO} that the above integral is well defined for all $x\in\R$, and also the representation of $Ai_\alpha$ in the form of a power series was presented.

Considering the fractional equation \eqref{8.1} with the random initial condition $u(0,x)=\eta(x)$, with $\eta$ being the same as in our exposition in the previous sections, we write down  the mean square solution in the form
\begin{equation*}
u(t, x) = \int_\mathbb{R} g_\alpha(t, x - y)\eta(y) dy,
\end{equation*}
and can state the same results on the modulus of continuity of solution and the bounds for the distribution of supremum as in Section 6. We just need to change the main assumption concerning the spectral measure $F$ of the initial data process $\eta$, namely, the condition \eqref{4.8} should be changed for the following one
\begin{equation*}
\int_\mathbb{R} \lambda^{2\alpha\beta} F(d \lambda) < \infty,
\end{equation*}
and the corresponding constant $c(\beta)$ appearing in the estimates \eqref{4.9}, \eqref{sup} takes the form $c(\beta) = 2^{1-\beta}\Big( \int_\mathbb{R} (\lambda + \lambda^{\alpha})^{2\beta} F(d\lambda)\Big)^{1/2}$.

\end{document}